\documentclass[a4paper,11pt]{amsart}
 \usepackage{float}
\usepackage{amssymb}
\usepackage{mathrsfs}
\usepackage{amsmath,amssymb,amsthm,latexsym,amscd,mathrsfs}
\usepackage{indentfirst}
\usepackage{stmaryrd}
\usepackage{graphicx}
\usepackage{extarrows}
\usepackage{multirow}
\usepackage{latexsym}
\usepackage{amsfonts}
\usepackage{color}
\usepackage{pictexwd,dcpic}
\usepackage{graphicx}
\usepackage{psfrag}
\usepackage{hyperref}
\usepackage{comment}
\usepackage{enumerate}
\usepackage{mathrsfs}

\numberwithin{equation}{section} \theoremstyle{plain}
\newtheorem{thm}{Theorem}[section]

\newtheorem{lem}[thm]{Lemma}

\newtheorem{conj}[thm]{Conjecture}

\newtheorem*{acknow}{Acknowledgments}
   
 \makeatletter

\def\<{\langle}
\def\>{\rangle}
\def\({\left(}
\def\){\right)}
\def\[{\left[}
\def\]{\right]}

\makeatother

\title{A note on the Chern Conjecture in dimension four}

\author[F.G. Li]{Fagui Li}
\address{School of Mathematical Sciences, Laboratory of Mathematics and Complex Systems, Beijing Normal University, Beijing 100875, P.R. CHINA.}
\email{faguili@mail.bnu.edu.cn}

\subjclass[2010]{15A45, 15B57, 53C42.}
\date{}
\keywords{Chern Conjecture, minimal hypersurfaces,  scalar curvature,  spheres.}

\begin{document}
\maketitle

\begin{abstract}
Let $M^4$  be a  closed immersed  minimal  hypersurface with constant squared length of the second fundamental form $S$ and constant 3-mean curvature  $H_3$ in $\mathbb{S}^{5}$. If
$H_3^2\leq \frac{1}{2}
$ and  Gauss-Kronecker curvature  $K_M$  satisfies  $K_M\leq1$ $($or $  K_M\leq\frac{S^2}{144}$$)$,  then $M^4$  is isoparametric.
\end{abstract}

\section{Introduction}
More than 50 years ago, S. S. Chern \cite{Chern68, Chern do Carmo Kobayashi 1970}  proposed the following famous  and original   conjecture:
\begin{conj} \label{Conjecture SSChern 1965}
Let $M^n$ be a  closed  immersed  minimal  hypersurface of the unit sphere $\mathbb{S}^{n+1}$ with constant scalar curvature $R_M$. Then for each $n$, the set of all possible values for $R_M$  is discrete.
\end{conj}
With the development of the study, mathematicians  realized the importance of Conjecture \ref{Conjecture SSChern 1965} and  proposed the following stronger version. Up to now, it is so far from a complete solution of this problem and S. T. Yau raised it again as the 105th problem in his Problem Section \cite{Yau  1982}.  Please see the excellent and detailed surveys on this topic 
by Scherfner-Weiss \cite{M Scherfner S Weiss 2008}, Scherfner-Weiss-Yau \cite{M Scherfner S Weiss and  Yau 2012} and Ge-Tang \cite{Ge Tang 2012}.
\begin{conj} \label{Conjecture strong version of Chern Conjecture}{\rm\textbf{$($Chern Conjecture$)$}}
Let $M^n$ be a  closed  immersed minimal hypersurface of the unit sphere $\mathbb{S}^{n+1}$ with constant scalar curvature. Then $M^n$ is isoparametric.
\end{conj}
The problem of classification for isoparametric hypersurfaces in spheres began in 1930 by Cartan 
and was finished by many mathematicians until 2020 (cf. Cecil-Chi-Jenson \cite{CCJ07}, Chi \cite{Chi11,Chi13,Chi16}, Dorfmeister-Neher \cite{Dorfmeister and  Neher 1985}, Immervoll \cite{Immervoll 2008} and Miyaoka \cite{Miy13,Miy16}, etc.), please see the elegant book \cite{CR15} and  survey \cite{Chi19} for more details.
 In 1968, J. Simons \cite{Simons68}  showed the following  theorem: 
\begin{thm}{\rm\textbf{$($Simons inequality$)$}}\label{thm J. Simons}
Let $M^n$ 
be a  closed immersed minimal hypersurface of the unit sphere $\mathbb{S}^{n+1}$  with the squared length of the second fundamental form $S$. 
Then
$$
\int_{M}S\left( S-n\right)\geq0.
$$
In particular, if $0\leq S\leq n$, one has either  $S\equiv0$ or $S\equiv n$  on $M^n$.
\end{thm}
 The classification
  of  $S\equiv n$  in Theorem \ref{thm J. Simons}  was characterized by Chern-do Carmo-Kobayashi \cite{Chern do Carmo Kobayashi 1970} and Lawson \cite{Lawson 1969}  independently:  \emph{The Clifford tori are the only closed minimal hypersurfaces in $\mathbb{S}^{n+1}$  with 
  $S \equiv n$, i.e.,
  $
  S^{k}(\sqrt{\frac{k}{n}})\times S^{n-k}(\sqrt{\frac{n-k}{n}}), \ \  1\leq k\leq n-1.
  $}
  For a closed immersed minimal hypersurface in $\mathbb{S}^{n+1}$,
  notice that 
  $$
  S=n(n-1)-R_M,
  $$
 by  the Gauss and Codazzi equations.
Hence, Simons inequality gave the first pinching gap of Conjecture \ref{Conjecture SSChern 1965}.

 In 1983, Peng and Terng \cite{Peng and Terng 1983,Peng and Terng1 1983} made the first breakthrough towards Chern  Conjecture \ref{Conjecture SSChern 1965}, they proved:
\emph{If $S > n$, then $S > n+ \frac{1}{12n}$. Moreover, for $n = 3$,  $S \geq  6$ if $S > 3$. } 
In 1993, Chang \cite{Chang 1993}  completed the proof of Chern  Conjecture \ref{Conjecture strong version of Chern Conjecture} for $n=3$. Next,
  Yang-Cheng \cite{Cheng Qing Ming and H Yang 1998}  and  
  Suh-Yang  \cite{Suh Yang2007}  improved the second pinching constant from $\frac{1}{12n}$ to $\frac{3n}{7}$. However, it is still an open problem for higher dimensional case that if $S\equiv {\rm Constant} > n$, then $S \geq 2n$? 
  
If  $S\not\equiv {\rm Constant}$, then the problem becomes more difficult.
 Peng and Terng \cite{Peng and Terng 1983,Peng and Terng1 1983} obtained 
 that 
 \emph{  there exists a positive constant 
   $\delta(n)$ depending only on $n$, 
   such that if $n\leq S\leq n+\delta(n)$, $n\leq 5$, 
   then $S\equiv n$.}
Later, Cheng and Ishikawa \cite{Cheng Qing Ming and Ishikawa1999} improved the previous pinching constant for $n\leq 5$,  Wei-Xu \cite{Si Ming Wei  Hong Wei Xu2007} extended the result to $n = 6, 7$ and  Zhang \cite{Qin Zhang 2010}  promoted it to $n\leq8$. 
Finally,   Ding-Xin \cite{Qi Ding  Y.L. Xin. 2011} proved  all the dimensions, in particular,  they showed that if  
the dimension is $n\geq6$, then the pinching constant $\delta(n)=\frac{n}{23}$.  Afterthat,  Xu-Xu \cite{Xu H. Xu 2017} improved it to $\delta(n)=\frac{n}{22}$ and  Li-Xu-Xu \cite{Li Lei Hongwei Xu  Zhiyuan Xu 2017} showed $\delta(n)=\frac{n}{18}$.  Actually, due to some counterexamples of Otsuki \cite{Otsuki 1970}, 
  the condition $S\geq n$ is essential in the pinching results above. 
 Very recently,  using the height functions of the normal vector field (cf. \cite{Ge Li 2020 Integral Einstein,Nomizu and Smyth 1969}),  Ge-Li \cite{Ge Li 2020 A lower bound second fundamental form} proved that \emph{ there is a positive constant
  $\delta(n)>0$ depending only on $n$ such that on any closed embedded, non-totally geodesic, minimal hypersurface $M^n$ in $\mathbb{S}^{n+1}$,
 $\int_{M}S \geq \delta(n){\rm Vol}(M^n)$}.

Lately,  de Almeida-Brito-Scherfner-Weiss \cite{Almeida  Brito Scherfner and  Weiss 2018} showed that   \emph{ $M^n$ $(n\geq4)$  is isoparametric if
it is  a closed, minimally immersed hypersurface of $\mathbb{S}^{n+1}$  with constant Gauss-Kronecker curvature and it has three pairwise distinct principal curvatures everywhere. }
 For the case that $n=4$,  Tang and Yang \cite{Tang and Yang 2018} proved that,  \emph{ if $R_M\geq0 $,  $H_3$ and  the number of distinct principal curvatures $g$ are constant,  then $M^4$ is isoparametric.}
 Deng-Gu-Wei \cite{Deng Gu and  Wei  2017} proved that if  \emph{ $M^4$ is a closed Willmore minimal hypersurfaces  with
constant scalar curvature in $\mathbb{S}^5$, then it is isoparametric.}  In other words, they dropped the non-negativity assumption of the scalar curvature  under the condition  $H_3\equiv0$.  

A recent great  progress of  
 Tang-Wei-Yan  \cite{TWY18} and  Tang-Yan \cite{TY20} generalized the theorem of de Almeida and Brito \cite{Almeida and Brito 1990} for $n = 3$ to any dimension $n$, strongly supporting Chern Conjecture \ref{Conjecture strong version of Chern Conjecture}.
  Note that the scalar curvature $R_M\geq0$ for all isoparametric hypersurfaces  and it can be found in 
  \cite{TY20}. 
\begin{thm}{\rm\textbf{$($Tang and Yan \label{Theorem Tang Yan ISOPARAMETRIC HYPERSURFACES non-negative scalar curvature}\cite{TY20}$)$}}
Let $M^n\ (n\geq4) $ be a 
closed immersed hypersurface  in  $\mathbb{S}^{n+1}$. If the following conditions are satisfied:
\begin{itemize}
\item[(i)]  $\sum_{i=1}^{n}\lambda_i^k$
$(k=1,\cdots,n-1)$ are constants for principal curvatures $\lambda_1,\lambda_2,\cdots,\lambda_n$;
\item[(ii)] $R_M\geq0$;
\end{itemize}
then $M^n$ is isoparametric. Moreover, if $M^n$ has $n$ distinct principal curvatures somewhere, then $R_M \equiv0$.
\end{thm}
As an application of Theorem \ref{Theorem Tang Yan ISOPARAMETRIC HYPERSURFACES non-negative scalar curvature} in dimension four, 
we remove the condition of the scalar curvature $R_M\geq 0$, but we have some requirements for the  Gauss-Kronecker curvature $K_M$ and 3-mean curvature $H_3$. 
\begin{thm}\label{thm introduction Chern conjecture for 4-dim}
Let $M^4$  be a  closed immersed  minimal  hypersurface with constant scalar curvature $R_M$ and constant 3-mean curvature  $H_3$ in $\mathbb{S}^{5}$. If
$H_3^2\leq\frac{1}{2}
$ and Gauss-Kronecker curvature  $K_M$  satisfies   $  K_M\leq1$ $($or $   K_M\leq\frac{S^2}{144}$$)$,  then $M^4$   is isoparametric.
\end{thm}

\section{Proof of Theorem \ref{thm introduction Chern conjecture for 4-dim}} 
In this section, we will prove Theorem \ref{thm introduction Chern conjecture for 4-dim}. 
Let $M^n$  be a  closed immersed  minimal  hypersurface  in the unit sphere $\mathbb{S}^{n+1}$ and denote by  $h$  the second fundamental
form of  hypersurface with respect to the unit normal vector field  $\nu$.
If $\left\lbrace \omega_1, \omega_2, \omega_3, \omega_4\right\rbrace$ is a smooth orthonormal coframe field, then $h$ can be written as
$$
h=\sum_{i,j}h_{ij}\omega_i\otimes\omega_j.
$$
The covariant derivative $\nabla h$ with components $h_{ijk}$ is given by
$$
\sum_{k}h_{ijk}\omega_k=dh_{ij}+\sum_{k}h_{kj}\omega_{ik}+\sum_{k}h_{ik}\omega_{jk},
$$
and  $\left\lbrace \omega_{ij}\right\rbrace $
is the connection forms of $M^4$ with respect to  $\left\lbrace \omega_1, \omega_2, \omega_3, \omega_4\right\rbrace $, which  satisfy the following structure equations:
$$
d\omega_i=-\sum_{j}\omega_{ij}\wedge\omega_j,\ \  \omega_{ij}+\omega_{ji}=0,
$$ 
$$
d\omega_{ij}=-\sum_{k}\omega_{ik}\wedge\omega_{kj}+\frac{1}{2}\sum_{k,l}R_{ijkl}\omega_{k}\wedge\omega_{l},
$$
where $R_{ijkl}$ is  the coefficients of the Riemannian curvature tensor on $M^4$.  We have the Gauss and Codazzi euqations:
$$
R_{ijkl}=\delta_{ik}\delta_{jl}-\delta_{il}\delta_{jk}+h_{ik}h_{jl}-h_{il}h_{jk},
$$
and 
$$
h_{ijk}=h_{ikj}.
$$
It is a well-known fact that the dual $(1,1)$ tensor $A$ (shape operator) of  $h$  is a self-adjoint linear operator in each tangent plane $T_pM$ and  its eigenvalues $\lambda_1(p), \lambda_2(p),\dots, \lambda_n(p)$ are the principal curvatures.  
 Associated to the shape operator $A$ there are $n$  algebraic invariants given by
$$
\sigma_r(p) = \sigma_r(\lambda_1(p), \lambda_2(p),\dots, \lambda_n(p)),\ \ 1 \leq r \leq n,
$$
where $\sigma_r : \mathbb{R} ^n \to \mathbb{R} $ is the elementary symmetric functions in $ \mathbb{R} $ given by
$$
\sigma_r(x_1,\cdots,x_n)=\sum_{i_1<i_2<\dots< i_r}x_{i_1} x_{i_2}\cdots x_{i_r}.
$$
Observe that the characteristic polynomial of $A$ can be writen as
$$
\det(\lambda I_n - A) =\sum_{r=0}^{n}
(-1)^r\sigma_r \lambda^{n-r}.
$$
The $r$-mean curvature $H_r$ of the hypersurface is then defined by 
\begin{equation}\label{equation H and sigma}
\binom{n}{r} H_r = \sigma_r.
\end{equation}
Suppose 
$$
f_k={\rm Tr}(A^k),
$$
by $n=4$,  $H_1=0$, $\sigma_4=K_M$ and $f_2={\rm Tr}(A^2)=\|h\|^2=S$ we have
\begin{eqnarray} \label{equation f1234}
\left\{ 
\begin{array}{lll} 
f_1 &=& \sigma_1=nH_1=0 \\ 
f_2 &=& \sigma_1^2-2\sigma_2=S\\
f_3 &=& \sigma_1^3-3\sigma_1\sigma_2+3\sigma_3=3\sigma_3\\
f_4 &=& \sigma_1^4-4\sigma_1^2\sigma_2+4\sigma_1\sigma_3+2\sigma_2^2-4\sigma_4=\frac{S^2}{2}-4K_M.\\
\end{array} 
\right. 
\end{eqnarray} 

\begin{lem}\label{lem  4 distinct principal curvatures}
Let $M^4$ be a  closed immersed minimal  hypersurface in $\mathbb{S}^{5}$ with constant scalar curvature $R_M\not=6$ and  constant 3-mean curvature $H_3$ $($or equivalently  $f_3$ is constant$)$. If there are  4 distinct principal curvatures at the minimmum point and maximum point of $K_M$, then Gauss-Kronecker curvature $K_M$  satisfies
\begin{equation}\label{equation KM 4 distinct principal curvatures}
 \sup_{x\in M^4}K_M(x)\geq 
 \frac{S^2(S-10)+6f_3^2}{48(S-6)}\geq 
\inf_{x\in M^4}K_M(x).
\end{equation}
\end{lem}
\begin{proof}
Set $x_{max}\in M^4$ and $x_{min}\in M^4$ such that
$$
K_M(x_{max})=\sup_{x\in M^4}K_M(x),\ \ K_M(x_{min})=\inf_{x\in M^4}K_M(x).
$$
At point $p$ ($p=x_{max}$ or  $p=x_{min}$),  we can take orthonormal frames such that
$h_{ij}=\lambda_i\delta_{ij}$ 
for all $i,j$. Thus at this point, we have
\begin{eqnarray*} 
\left\{ 
\begin{array}{lll} 
\sum_{i=1}^{4}h_{iik} &= 0 &\\ 
\sum_{i=1}^{4}\lambda_ih_{iik} &=0& \\
\sum_{i=1}^{4}\lambda_i^2h_{iik} &=0& \\
\sum_{i=1}^{4}\lambda_i^3h_{iik} &=0.&
\end{array} 
\right. 
\end{eqnarray*} 

The first, second and third   equations hold because $f_1, f_2$ and $f_3$ are constant. The fourth one comes from the fact that $p$ is an extreme point of $K_M$. Then $h_{iik}=0$ by $\lambda_i\neq\lambda_j$  $\left( i\neq j\right) $ at $p$. 
Since $f_3$ is constant and due to Peng and Terng \cite{Peng and Terng 1983,Peng and Terng1 1983}, one has
\begin{equation}\label{A-2B}
\mathscr{A}-2\mathscr{B}=Sf_4-f_3^2-S^2,
\end{equation}
and
\begin{equation}\label{laplacian of f4}
\frac{1}{4}\Delta f_4
= (4-S)f_4+2\mathscr{A}+\mathscr{B},
\end{equation}
where 
$$\mathscr{A}=\sum_{i,j,k}h_{ijk}\lambda_i^2,\ \  \mathscr{B}=\sum_{i,j,k}h_{ijk}\lambda_i\lambda_j.$$
In addition, due to $S$ is constant and by Simons' identity \cite{Simons68}  we obtain \begin{equation}\label{equation Simons equation 1}
0=\frac{1}{2}\Delta S=|\nabla h|^2+S(4-S),
\end{equation}
where 
$|\nabla h|^2=\sum_{i,j,k}h_{ijk}^2$.
 Since $h_{iik}=0$ for all $i,k$ at $p$ and  let 
 $$
 C=\lambda_1^2h_{234}^2+\lambda_2^2h_{134}^2+\lambda_3^2h_{124}^2+\lambda_4^2h_{123}^2,
 $$
  we can directly calculate
\begin{equation}\label{3A-6B}
\begin{aligned}
3(\mathscr{A}-2\mathscr{B})
&=\sum_{i,j,k}h_{ijk}^2\left(\left( \lambda_i^2+\lambda_j^2+\lambda_k^2\right) -2\lambda_i\lambda_j-2\lambda_i\lambda_k-2\lambda_j\lambda_k \right) \\
&=\sum_{i,j,k}h_{ijk}^2\left(2\left( \lambda_i^2+\lambda_j^2+\lambda_k^2\right) -\left(\lambda_i+\lambda_j+\lambda_k \right) ^2 \right) \\
&=6\left( h_{123}^2(2S-3\lambda_4^2)+h_{124}^2(2S-3\lambda_3^2)+h_{134}^2(2S-3\lambda_2^2)+h_{234}^2(2S-3\lambda_1^2)\right) \\
&=2S|\nabla h|^2-18C\\
&=2S^2(S-4)-18C.
\end{aligned}
\end{equation}
Similarly
\begin{equation}\label{A}
\begin{aligned}
\mathscr{A}
&=\frac{1}{3}\sum_{i,j,k}h_{ijk}^2\left( \lambda_i^2+\lambda_j^2+\lambda_k^2\right)  \\
&=2\left( h_{123}^2(S-\lambda_4^2)+h_{124}^2(S-\lambda_3^2)+h_{134}^2(S-\lambda_2^2)+h_{234}^2(S-\lambda_1^2)\right) \\
&=\frac{1}{3}S|\nabla h|^2-2C\\
&=\frac{1}{3}S^2(S-4)-2C.
\end{aligned}
\end{equation}
By (\ref{3A-6B}) and (\ref{A}), we have
\begin{equation}\label{B}
\mathscr{B}=-\frac{1}{6}S^2(S-4)+2C.
\end{equation}
Due to  $f_4 =\frac{S^2}{2}-4K_M$ by (\ref{equation f1234}) and 
 (\ref{A-2B})-(\ref{B}),  we have
\begin{equation*}
18\Delta K_M(p)+48K_M(p)(S-6)=S^2(S-10)+6f_3^2.
\end{equation*}
The maximum principle implies that
$$
\Delta K_M(x_{max})\leq0,\ \ 
\Delta K_M(x_{min})\geq0.
$$
Hence
\begin{equation*}\label{equation isopermetric KMax} 
48K_M(x_{max})(S-6)\geq S^2(S-10)+6f_3^2
\geq 
48K_M(x_{min})(S-6).
\end{equation*}
 Specially, 
  if  $0\leq S< 6$, then $K_M(x_{max})\leq K_M(x_{min})$, we have $K_M$ is constant and $M^4$ is isoparametric, i.e., $S=0$ or $S=4$. 
The proof is complete. 
\end{proof}

\begin{lem}\label{lem  3 distinct principal curvatures}
Let $M^4$ be a  closed immersed minimal  hypersurface in $\mathbb{S}^{5}$ with constant scalar curvature $R_M$ and  constant 3-mean curvature $H_3$ $($or equivalently  $f_3$ is constant$)$. If there exists a point $p\in M^4$ with three distinct principal curvatures, then Gauss-Kronecker curvature $K_M$  satisfies
\begin{equation}\label{equation Delta K_M 3 distinct principal curvature}
-\Delta K_M(p)=
4(S-4)K_M(p)-2C_1+2(6\lambda^2-S)\left(h_{111}^2+h_{112}^2 \right),
\end{equation}
where 
$C_1=\lambda_1^2h_{234}^2+\lambda_2^2h_{134}^2+\lambda_3^2\left( h_{124}^2+h_{114}^2\right) +\lambda_4^2\left( h_{123}^2+h_{113}^2\right) $,
 $\lambda_1(p)=\lambda_2(p)=\lambda$,  $\lambda_3(p)=\mu-\lambda$ and $\lambda_4(p)=-\mu-\lambda$.
\end{lem}
\begin{proof}
At point $p\in M^4$,  we can take orthonormal frames such that
$h_{ij}=\lambda_i\delta_{ij}$ 
for all $i,j$. Thus at this point, we have
\begin{eqnarray*} 
\left\{ 
\begin{array}{lll} 
\sum_{i=1}^{4}h_{iik} &= 0 &\\ 
\sum_{i=1}^{4}\lambda_ih_{iik} &=0& \\
\sum_{i=1}^{4}\lambda_i^2h_{iik} &=0&.
\end{array} 
\right. 
\end{eqnarray*} 

The first, second and third   equations hold because $f_1, f_2$ and $f_3$ are constant.
Then  for all $1\leq k\leq 4$, one has
\begin{equation}\label{equation hiik 3 distinct principal curvatures}
h_{11k}=-h_{22k},\  \  h_{33k}=h_{44k}=0,
\end{equation}
by $\lambda_i\neq\lambda_j$  $\left( 2\leq i\neq j \leq4 \right) $ at $p$. 
Let $
C_2=h_{234}^2+h_{123}^2+ h_{124}^2+h_{134}^2+h_{114}^2 +h_{113}^2 $, by (\ref{equation hiik 3 distinct principal curvatures})  we have
\begin{equation}\label{equation 3A-6B 3 distinct principal curvatures}
\begin{aligned}
&\ \ \ \ 3(\mathscr{A}-2\mathscr{B})\\
&=\sum_{i,j,k}h_{ijk}^2\left(2\left( \lambda_i^2+\lambda_j^2+\lambda_k^2\right) -\left(\lambda_i+\lambda_j+\lambda_k \right) ^2 \right) \\
&=6\left( h_{123}^2(2S-3\lambda_4^2)+h_{124}^2(2S-3\lambda_3^2)+h_{134}^2(2S-3\lambda_2^2)+h_{234}^2(2S-3\lambda_1^2)\right)+ \\
&\ \ \ \ 
h_{111}^2(6\lambda^2-9\lambda^2)+
3\left(h_{112}^2(6\lambda^2-9\lambda^2) +h_{113}^2(2S-3\lambda_4^2)
+h_{114}^2(2S-3\lambda_3^2)\right)+ \\
&\ \ \ \
h_{222}^2(6\lambda^2-9\lambda^2)+
3\left(
h_{221}^2(6\lambda^2-9\lambda^2) +h_{223}^2(2S-3\lambda_4^2)
+h_{224}^2(2S-3\lambda_3^2)\right)
\\
&=12SC_2-18C_1-12\lambda^2\left( h_{111}^2+h_{112}^2\right),
\end{aligned}
\end{equation}
and
\begin{equation}\label{equation A 3 distinct principal curvatures}
\begin{aligned}
\mathscr{A}
&=\frac{1}{3}\sum_{i,j,k}h_{ijk}^2\left( \lambda_i^2+\lambda_j^2+\lambda_k^2\right)  \\
&=2\left( h_{123}^2(S-\lambda_4^2)+h_{124}^2(S-\lambda_3^2)+h_{134}^2(S-\lambda_2^2)+h_{234}^2(S-\lambda_1^2)\right) +\\
&\ \ \ \ 
\lambda^2h_{111}^2+3\lambda^2h_{112}^2+
h_{113}^2\left( S-\lambda_4^2\right) +
h_{114}^2\left( S-\lambda_3^2\right) +\\
&\ \ \ \ 
\lambda^2h_{222}^2+3\lambda^2h_{221}^2+
h_{223}^2\left( S-\lambda_4^2\right) +
h_{224}^2\left( S-\lambda_3^2\right) \\
&=2SC_2 -2C_1+4\lambda^2\left( h_{111}^2+h_{112}^2\right).
\end{aligned}
\end{equation}
By (\ref{equation 3A-6B 3 distinct principal curvatures}) and (\ref{equation A 3 distinct principal curvatures}), we have
\begin{equation}\label{equation B 3 distinct principal curvatures}
\mathscr{B}=-SC_2+2C_1+4\lambda^2\left(h_{111}^2 +h_{112}^2\right) .
\end{equation}
By (\ref{equation Simons equation 1}) and (\ref{equation hiik 3 distinct principal curvatures})  we obtain
 \begin{equation}\label{equation covariant derivative of h}
\begin{aligned}
|\nabla h|^2
&=S(S-4)=\sum_{i,j,k}h_{ijk}^2\\
&=6\left( h_{234}^2+h_{123}^2+ h_{124}^2+h_{134}^2\right) +
3\left( h_{112}^2 +h_{113}^2 +h_{114}^2 \right) +
\\
&\ \ \ \ 
3\left( h_{221}^2 +h_{223}^2 +h_{224}^2 \right) +h_{111}^2 +h_{222}^2 \\
&=6C_2+4\left(h_{111}^2 +h_{112}^2\right) .
\end{aligned}
\end{equation}
By (\ref{equation A 3 distinct principal curvatures}),  (\ref{equation B 3 distinct principal curvatures}) and (\ref{equation covariant derivative of h}),  we have
\begin{equation}\label{equation A no C2 3 distinct principal curvatures}
\mathscr{A}=\frac{1}{3}S^2(S-4)-2C_1+4\left( \lambda^2-\frac{S}{3}\right) \left(h_{111}^2 +h_{112}^2\right),
\end{equation}
and
\begin{equation}\label{equation B no C2 3 distinct principal curvatures}
\mathscr{B}=-\frac{1}{6}S^2(S-4)
+2C_1+2\left(2 \lambda^2+\frac{S}{3}\right) \left(h_{111}^2 +h_{112}^2\right) .
\end{equation}
Due to  (\ref{equation f1234}), (\ref{laplacian of f4}), (\ref{equation A no C2 3 distinct principal curvatures}) and (\ref{equation B no C2 3 distinct principal curvatures}), we have
$$
-\Delta K_M(p)=
4(S-4)K_M(p)-2C_1+2(6\lambda^2-S)\left(h_{111}^2+h_{112}^2 \right).
$$
The proof is complete.
\end{proof}

\begin{proof}[\textbf{Proof of Theorem $\mathbf{\ref{thm introduction Chern conjecture for 4-dim}}$}]
Suppose $K_M(p)=\sup_{x\in M^4}K_M(x)\in M^4$, we have just five possibilities for the principal curvatures $\lambda_1(p),\dots, \lambda_4(p)$ at  point $p\in M^n$:

\begin{enumerate} [(1)]
\item $\lambda_i(p)=\lambda_j(p)$ for all $1\leq i,j\leq4$.
\item $\lambda_1(p)=\lambda_2(p)=\lambda$ and $\lambda_3(p)=\lambda_4(p)=-\lambda$,
$$
A(p)=
\begin{pmatrix}
\lambda & 0 & 0 & 0 \\
 0 & \lambda & 0& 0\\
 0 & 0 &  -\lambda & 0\\
0 & 0& 0& -\lambda
 \end{pmatrix}.
 $$
 
 \item $\lambda_1(p)=\lambda_2(p)=\lambda_3(p)=\lambda$ and $\lambda_4(p)=-3\lambda$,
 $$
 A(p)=
 \begin{pmatrix}
 \lambda & 0 & 0 & 0 \\
  0 & \lambda & 0& 0\\
  0 & 0 &  \lambda & 0\\
 0 & 0& 0& -3\lambda
  \end{pmatrix}.
  $$
  
  \item $\lambda_1(p)=\lambda_2(p)=\lambda$,  $\lambda_3(p)=\mu-\lambda$ and $\lambda_4(p)=-\mu-\lambda$.
  $$
  A(p)=
  \begin{pmatrix}
  \lambda & 0 & 0 & 0 \\
   0 & \lambda & 0& 0\\
   0 & 0 & \mu-\lambda & 0\\
  0 & 0& 0& -\mu-\lambda
   \end{pmatrix}.
   $$
   
 \item $\lambda_i(p)\neq\lambda_j(p)$ for all $1\leq i\neq j\leq4$.
\end{enumerate}

Due to Theorem \ref{Theorem Tang Yan ISOPARAMETRIC HYPERSURFACES non-negative scalar curvature},   $M^4$ is isoparametric if $S\leq 12$. Hence,  we  just need to prove $S\leq 12$ at $p$, since $S$ is constant on $M^4$.

In the case  (1), $S\equiv 0$.

In the case  (2), due to $H_3$ is constant, also $f_3=12H_3$ is constant by (\ref{equation H and sigma}) and (\ref{equation f1234}), then 
$f_3\equiv0$,
$M^4$ is isoparametric (see Deng-Gu-Wei \cite{Deng Gu and  Wei  2017}). In fact,  $K_M\leq1$ (or $K_M\leq\frac{S^2}{144}$) implies that
$$
K_M=\lambda^4\leq1\ \ (or\ K_M=\lambda^4 \leq\frac{S^2}{144}),
$$
and
 $S=4\lambda^2\leq4$ (or $\lambda^4 \leq\frac{S^2}{144}=\frac{16\lambda^4}{144}$). Hence, $S\leq4$ in this case.

In the case  (3), 
if $f_3^2=576\lambda^6\leq 576$, then  $\lambda^6\leq 1$ and
$S=12\lambda^2\leq 12$.

In the  case  (4), 
some direct calculations show
\begin{eqnarray} \label{equation Sf3KM}
\left\{ 
\begin{array}{lll} 
S &= 4\lambda^2+2\mu^2 &\\ 
f_3&=-6\mu^2\lambda& \\
K_M &=\lambda^2(\lambda^2-\mu^2).&\\
\end{array} 
\right. 
\end{eqnarray} 
By (\ref{equation Sf3KM}), 
one has 
\begin{equation}\label{equation lambda mu}
\lambda^2=\frac{\mu^2+\sqrt{\mu^4+4K_M}}{2}\ \  or\ \ 
\lambda^2=\frac{\mu^2-\sqrt{\mu^4+4K_M}}{2}.
\end{equation}
If $K_M(p)=\lambda^2(\lambda^2-\mu^2)<0$, we have
$
0<\lambda^2<\mu^2.
$
The maximum principle implies 
$$
\Delta K_M(p)\leq0.
$$
Due to  (\ref{equation Delta K_M 3 distinct principal curvature}) by Lemma  \ref{lem  3 distinct principal curvatures} and $S>4$, one has
$$
0\leq-\Delta K_M(p)=
4(S-4)K_M(p)-2C_1+2(6\lambda^2-S)\left(h_{111}^2+h_{112}^2 \right),
$$
and
$$
0\leq C_1<(6\lambda^2-S)\left(h_{111}^2+h_{112}^2 \right)=2\left( \lambda^2-\mu^2 \right) \left(h_{111}^2+h_{112}^2 \right)\leq 0.
$$
This creates a contradiction.
Thus, $K_M(p)\geq0$.  By  (\ref{equation lambda mu}), we obtain
$$
f_3^2=36\mu^4\lambda^2=18\mu^4\left( \mu^2+\sqrt{\mu^4+4K_M}\right)\geq36\mu^6.
$$
If $f_3^2\leq\frac{576\sqrt{10}}{25}$ and  $ K_M\leq1$ $($or $  K_M\leq\frac{S^2}{144}$$)$,  then
$$
\begin{aligned}
S
&=4\lambda^2+2\mu^2= 4\mu^2+2\sqrt{\mu^4+4K_M}\leq
2\sqrt{10(\mu^4+2K_M)}\\
&\leq2\sqrt{10}\sqrt{\left( \frac{f_3^2}{36}\right)^{\frac{2}{3}} +2}\\
&=12,
\end{aligned}
$$
(or 
$
S\leq2\sqrt{10}\sqrt{\left( \frac{f_3^2}{36}\right)^{\frac{2}{3}} +\frac{S^2}{72}}\leq\sqrt{64+\frac{5S^2}{9}}
$ shows that $S\leq 12$).

In the case (5),  by   (\ref{equation KM 4 distinct principal curvatures}) in  Lemma  \ref{lem  4 distinct principal curvatures} and $K_M\leq1$ $($or $  K_M\leq\frac{S^2}{144}$$)$,  we have 
$$
\left(or\ \frac{S^2}{144}\geq  \right) 
1\geq K_M(p)\geq 
 \frac{S^2(S-10)+6f_3^2}{48(S-6)}\geq 
  \frac{S^2(S-10)}{48(S-6)},
$$
and it implies that $S\leq 12$ if $S\not= 6$.

To sum up, all the cases show that $S\leq 12$ if
$f_3^2\leq\min\{576,\frac{576\sqrt{10}}{25}\}=\frac{576\sqrt{10}}{25}$. By (\ref{equation H and sigma}) and (\ref{equation f1234}),
we have $$H_3^2=\frac{\sigma^2_3}{16}=\frac{f_3^2}{144}\leq\frac{4\sqrt{10}}{25}\approx0.5059.
$$
This completes the proof by $H_3^2\leq 0.5<0.5059$.
\end{proof}

\begin{acknow}
The author is very grateful to Professor Jianquan Ge, Professor Wenjiao Yan  and Dr. Qichao Li for their kindly encouragements and supports. 
\end{acknow}

\end{document}